\documentclass[a4paper,reqno]{amsart}

\usepackage[utf8]{inputenc}
\usepackage{amssymb}
\usepackage{amsmath}
\usepackage{hyperref}
\usepackage{tikz-cd}
\usepackage{mathtools}

\newtheorem{theorem}{Theorem}[section]
\newtheorem*{theorem-non}{Theorem}
\newtheorem{lemma}[theorem]{Lemma}
\newtheorem{proposition}[theorem]{Proposition}
\newtheorem{corollary}[theorem]{Corollary}

\theoremstyle{definition}
\newtheorem{definition}[theorem]{Definition}

\newtheorem{remark}[theorem]{Remark}

\newcommand{\CC}{\ensuremath{\mathbb{C}}}

\newcommand{\PP}{\ensuremath{\mathbb{P}}}
 
\newcommand{\RR}{\ensuremath{\mathbb{R}}} 

\newcommand{\ZZ}{\ensuremath{\mathbb{Z}}}

\newcommand{\cE}{\mathcal{E}}

\newcommand{\cM}{\mathcal{M}}

\newcommand{\cO}{\mathcal{O}}

\newcommand{\cU}{\mathcal{U}}

\newcommand{\cZ}{\mathcal{Z}}

\newcommand{\fS}{\mathfrak{S}}

\DeclareMathOperator{\ch}{ch}

\DeclareMathOperator{\Coh}{Coh}
\DeclareMathOperator{\coker}{coker}

\DeclareMathOperator{\Ext}{Ext}

\DeclareMathOperator{\Hom}{Hom}

\renewcommand{\Im}{\mathrm{Im}}

\DeclareMathOperator{\NS}{NS}

\DeclareMathOperator{\pt}{pt}

\DeclareMathOperator{\rk}{rk}

\DeclareMathOperator{\Stab}{Stab}

\newcommand\da{\widehat{a}}
\DeclareMathOperator\IZ{\mathcal{I}_{\mathcal{Z}}}

\DeclareMathOperator\Db{D^b}
\newcommand\An{A^{[n+1]}}
\newcommand\dA{\widehat{A}}
\newcommand\ddA{\skew{5.5}\widehat{\widehat{A}}}
\DeclareMathOperator\Pb{\mathcal{P}}
\DeclareMathOperator{\SStab}{SStab}

\newcommand{\KA}{K_n(A)}

\title{Stable vector bundles on generalized Kummer varieties}

\author{Fabian Reede}
\address{Institut f\"ur Algebraische Geometrie, Leibniz Universit\"at Hannover, Welfengarten 1, 30167 Hannover, Germany}
\email{reede@math.uni-hannover.de}

\author{Ziyu Zhang}
\address{Institute of Mathematical Sciences, ShanghaiTech University, 393 Middle Huaxia Road, 201210 Shanghai, P.R.China}
\email{zhangziyu@shanghaitech.edu.cn}

\keywords{stable sheaves, moduli spaces, generalized Kummer varieties}

\makeatletter
\@namedef{subjclassname@2020}{%
  \textup{2020} Mathematics Subject Classification}
\makeatother

\subjclass[2020]{Primary: 14J60; Secondary: 14F08, 14D20, 53C26}

\begin{document}

\begin{abstract}
	For an abelian surface $A$, we explicitly construct two new families of stable vector bundles on the generalized Kummer variety  $\KA$ for $n\geqslant 2$. The first is the family of tautological bundles associated to stable bundles on $A$, and the second is the family of the ``wrong-way'' fibers of a universal family of stable bundles on the dual abelian surface $\widehat{A}$ parametrized by $\KA$. Each family exhibits a smooth connected component in the moduli space of stable bundles on $\KA$, which is holomorphic symplectic but not simply connected, contrary to the case of K3 surfaces.
\end{abstract}

\maketitle

\section*{Introduction}

\subsection*{Background}

Irreducible holomorphic symplectic manifolds are a type of building blocks in the classification of compact K\"ahler manifolds with trivial first Chern class. In the very influential paper \cite{beau}, Beauville constructed two classes of irreducible holomorphic symplectic manifolds, which are the Hilbert schemes $X^{[n]}$ of $n$-points on K3 surfaces $X$, and the generalized Kummer varieties $K_n(A)$ associated to abelian surfaces $A$, obtained as the zero fibers of the summation map $\Sigma: A^{[n+1]} \rightarrow A$. The second construction was later generalized by Yoshioka in \cite{yosh}, in which he proved that the fibers $K_H(v)$ of the Albanese morphism $\mathfrak{a}_v : M_H(v) \rightarrow A\times \dA$ for moduli spaces $M_H(v)$ of $\mu_H$-stable sheaves on $A$ with Mukai vector $v$ are deformation equivalent to generalized Kummer varieties.

\subsection*{Main results}

The present manuscript is a continuation of the authors' work \cite{rz,rz2} on the construction of new stable sheaves on irreducible holomorphic symplectic manifolds. The same problem was also studied by various authors, such as in \cite{stapleton, wandel,wray}. Recently, Markman \cite{markman} and O'Grady \cite{ogrady} also found examples of stable bundles among modular sheaves on K3$^{[n]}$'s. The main purpose of this manuscript is to construct new stable vector bundles on generalized Kummer varieties and study some of their properties. We achieved two different constructions.

A natural family of vector bundles on $K_n(A)$ for considering stability are the so-called tautological bundles. In \cite{wandel} Wandel constructed some examples of tautological bundles on $K_n(A)$ for $n=1,2$. Following an idea of Stapleton \cite{stapleton}, we generalize Wandel's results by proving that in fact all tautological bundles on $K_n(A)$ for $n\geqslant 2$ are stable with respect to suitable ample classes. Moreover, under suitable numerical assumptions, we show that the tautological bundles form a connected component of the moduli space of stable bundles on $K_n(A)$. This in particular indicates that smooth components of moduli spaces of sheaves on generalized Kummer varieties need not be irreducible holomorphic symplectic manifolds, contrary to the result for K3 surfaces, see \cite[Theorem 10.3.10]{huy2}.

For another family of vector bundles on $K_n(A)$, we use the standard Fourier-Mukai transform to construct a fine moduli space $M_{\widehat{H}}(w)$ of stable vector bundles of rank $r\geqslant n+2$ on the dual abelian variety $\dA$ for some suitable choice of the Mukai vector $w$, such that $M_{\widehat{H}}(w) \cong A^{[n+1]} \times \widehat{A}$. Then $K_n(A)$ is naturally isomorphic to the zero fiber of the Albanese morphism $\mathfrak{a}_w: M_{\widehat{H}}(w) \to A \times \widehat{A}$. Let $\cU$ be the restriction of the universal family on $\widehat{A} \times M_{\widehat{H}}(w)$ to the closed subscheme $\widehat{A} \times K_n(A)$. For each closed point $\da\in \dA$, the further restriction of $\cU$ to the slice $\{ \da \} \times K_n(A)$ gives a vector bundle $\mathcal{U}_{\da}$ on $K_n(A)$. Following our approach in \cite{rz2}, we show that each $\cU_{\da}$ is a stable bundle on $K_n(A)$, and hence we obtain a family of stable bundles on $K_n(A)$ parametrized by $\dA$.

Our main results can be summarized as follows:

\begin{theorem-non}
	Let $(A, H)$ be a polarized abelian surface, and let $(\dA, \widehat{H})$ be its dual.
	\begin{enumerate}
		\item (Theorem \ref{thm:firstmain}) Let $E$ be a $\mu_H$-stable vector bundle of class $v$ on $A$ with $E\neq \mathcal{O}_A$. Then the tautological bundle $E^{(n)}$ is a $\mu_D$-stable vector bundle on $K_n(A)$ with respect to some ample divisor $D$ on $K_n(A)$. Moreover, under suitable numerical assumptions on $v$, the moduli space $M_H(v)$ of $\mu_H$-stable vector bundles of class $v$ on $A$ can be embedded as a connected component of some moduli space of $\mu_D$-stable vector bundles on $K_n(A)$.
		\item (Theorem \ref{thm:component1}) Let $\cU$ be the restriction of the universal vector bundle on $\dA \times M_{\widehat{H}}(w)$ to the closed subscheme $\dA \times K_n(A)$ as described above. Then for each closed point $\da \in \dA$, the fiber $\cU_{\da}$ is a $\mu_D$-stable bundle on $K_n(A)$ with respect to an  ample divisor $D$. Moreover, $\dA$ can be embedded as a connected component of a moduli space of $\mu_D$-stable vector bundles on $K_n(A)$.
	\end{enumerate}
\end{theorem-non}

\subsection*{Sketch of proof}

Let us give a quick overview on how we achieved the above results. Although the setup in both cases looks very different, we will follow a similar chain of ideas to prove the slope stability of the bundles $E^{(n)}$ (resp. $\cU_{\da}$) with respect to some ample divisor $D$ on $K_n(A)$. The proof consists of the following three major steps.

\textsc{Step 1.} To begin with, let $P_n(A)$ be the codimension $2$ subvariety of $A^{n+1}$ parametrizing $(n+1)$-tuples whose components add up to zero under the group law of $A$. Each bundle $E^{(n)}$ (resp. $\cU_{\da}$) defines uniquely a reflexive sheaf on $P_n(A)$. We adapt the technique developed by Stapleton in \cite{stapleton} to prove the slope stability of $E^{(n)}$ (resp. $\cU_{\da}$) with respect to a natural nef divisor $H_K$ on $K_n(A)$ by showing that the corresponding reflexive sheaf cannot be destabilized by any $\mathfrak{S}_{n+1}$-invariant subsheaf on $P_n(A)$. (See Propositions \ref{thm:boundary-stable} and \ref{prop:stableEx}.)

\textsc{Step 2.} In order to show the slope stability of $E^{(n)}$ (resp. $\cU_{\da}$) with respect to an ample divisor on $K_n(A)$, we use the openness of stability to perturb $H_K$ to a nearby ample divisor $D$. This perturbation argument was developed in \cite{greb16,stapleton}, and generalized in \cite{rz2}. The main difficulty for our application is to show the existence of the ample divisor $D$ independent of the choice of $E$ in its own moduli $M_H(v)$ (resp. the choice of the fiber $\cU_{\da}$ in the family $\cU$ parametrized by $\dA$). (See Propositions \ref{prop:sameH2} and \ref{prop:sameH1}.)

\textsc{Step 3.} Finally, in order to identify $M_H(v)$ (resp. $\dA$) as a smooth connected component of the moduli space of $\mu_D$-stable sheaves on $K_n(A)$, we interpret $E^{(n)}$ (resp. $\cU_{\da}$) as the image of $E$ (resp. a line bundle on $A$) under an integral functor $\Theta$ induced by the structure sheaf (resp. the ideal sheaf) of the universal subscheme for $K_n(A)$. By a result of Meachan \cite{meac}, we can apply the technique of $\PP$-functors invented in \cite{add,add16-2} to compute the relevant cohomology groups, which leads to our conclusion. (See Theorems \ref{thm:firstmain} and \ref{thm:component1}.)

\subsection*{Structure of text}The text is organized in two sections, which respectively deal with the two cases mentioned above. All objects are defined over the field of complex numbers $\mathbb{C}$.

\subsection*{Acknowledgement} We thank the referee for carefully reading a previous version of the manuscript and several suggestions for improvement.

\section{Tautological Bundles}

\subsection{Notations}

For any integer $n \geqslant 2$, let $A$ be an abelian surface, let $\KA$ the generalized Kummer variety of dimension $2n$, and let $\cZ \subset A \times \KA$ the corresponding universal family. The projections from $\cZ$ to the two factors $A$ and $\KA$ are denoted by $p$ and $q$ respectively. The following diagram exhibits the relations of some relevant schemes:
\begin{equation}
	\label{eqn:Notations}
	\begin{tikzcd}
		P_n(A)_\circ \ar[r, "\sigma_\circ"] \ar[d, hook, "j_P"'] & S_n(A)_\circ \ar[d, hook, "j_S"'] & \KA_\circ \ar[l, "h_\circ"'] \ar[d, hook, "j_K"'] \\
		P_n(A) \ar[r, "\sigma"] \ar[d, hook, "\tau"'] & S_n(A) \ar[d, hook] & \KA \ar[l, "h"'] \ar[d, hook, "\iota"'] \\
		A^{n+1} \ar[r, "\overline{\sigma}"] & A^{(n+1)} & A^{[n+1]} \ar[l, "\overline{h}"']
	\end{tikzcd}
\end{equation}

Each vertical arrow in the lower half of the diagram is the embedding of a zero fiber of the addition morphism to $A$; $\sigma$ and $\overline{\sigma}$ are quotients by the symmetric group $\fS_{n+1}$; $h$ and $\overline{h}$ are Hilbert-Chow morphisms. Moreover, we denote the projections from $A^{n+1}$ to each individual factor by $q_0, q_1, \cdots, q_n$.

Each vertical arrow in the upper half of the diagram is the embedding of an open subscheme parametrizing $n+1$ distinct (ordered or unordered) points in $A$. It is clear that the complement of each of these embeddings is a closed subscheme of codimension $2$. The morphisms $\sigma_\circ$ and $h_\circ$ are restrictions of the morphisms in the second row. Clearly $\sigma_\circ$ is a free $\fS_{n+1}$-quotient and $h_\circ$ is an isomorphism.

Let $H$ be an ample divisor on $A$. For each $0 \leqslant i \leqslant n$, we define $h_i = \tau^\ast q_i^\ast H$. Then $H_P=\sum_{i=0}^n h_i$ on $P_n(A)$ is an $\fS_{n+1}$-invariant ample divisor on $P_n(A)$, which descends to an ample divisor $H_S$ on $S_n(A)$, whose pullback $H_K = h^\ast (H_S)$ is a big and nef divisor on $K_n(A)$. For any $E \in \Coh(A)$, the corresponding \emph{tautological sheaf} $E^{(n)}$ on $\KA$ is defined by $E^{(n)} = q_\ast p^\ast E$. Moreover, we write $E_i = \tau^\ast q_i^\ast E$ for each $0 \leqslant i \leqslant n$. The goal of this section is to show that if $E$ is a non-trivial $\mu_H$-stable vector bundle on $A$, then $E^{(n)}$ is slope stable with respect to some ample divisor sufficiently close to $H_K$. Our approach will mainly follow the idea in \cite{stapleton}.

\subsection{Pullback of stable bundles}

We aim to prove Proposition \ref{prop:StableE}, which is an analogue of \cite[Proposition 4.7]{stapleton} in the Kummer case. We first collect necessary notations and tools required in the course of the proof, following \cite[\S 4]{stapleton}.

For any normal projective variety $X$, let $\gamma \in N_1(X)_\RR$ be a curve class and $\cE \in \Coh(X)$ a torsion-free sheaf. The slope of $\cE$ with respect to $\gamma$ is defined by
$$ \mu^\gamma(\cE) = \frac{c_1(\cE) \cdot \gamma}{\rk \cE}. $$
It is clear that the slope is linear with respect to $\gamma$, and the usual notion of slope $\mu_H(\cE) = \mu^{H^{d-1}}(\cE)$ for any ample class $H \in N^1(X)_\RR$, where $n=\dim X$. The new notion of slope defines a \emph{slope stability} (resp. \emph{semistability}) of $\cE$ with respect to a curve class $\gamma$, by requiring any torsion-free quotient of $\cE$ of a smaller rank to have a smaller (resp. smaller or equal) slope with respect to $\gamma$.

The main advantage of studying slope stability with respect to curve classes is the linearity of slopes with respect to the curve parameter. More precisely, we have the following lemma.

\begin{lemma}[{\cite[Lemma 4.4]{stapleton}}]
	\label{lem:additive-stable}
	Let $\gamma, \delta \in N_1(X)_\RR$ such that $\cE$ is semistable with respect to $\gamma$ and stable with respect to $\delta$. Then $\cE$ is stable with respect to $a\gamma + b\delta$ for any $a,b>0$. \qed
\end{lemma}

Our main tool for determining the slope stability is the following observation

\begin{lemma}[{\cite[Corollary 4.6]{stapleton}}]
	\label{lem:slope-stable}
	Let $\pi \colon C_T \to T$ be a smooth family of irreducible closed curves in $X$ with class $\gamma$. Suppose that $\cE$ is a vector bundle on $X$ such that $\cE|_{C_t}$ is stable for all $t \in T$, and that the curves in $C_T$ are dense in $X$. Then $\cE$ is stable with respect to the curve class $\gamma$. Moreover, the statement also holds if stability is replaced by semistability.
	\qed
\end{lemma}

The following immediate implication will be convenient for our application

\begin{corollary}
	\label{cor:high-dim-fiber}
	Let $\pi \colon V_T \to T$ be a smooth projective family of irreducible closed subvarieties of $X$ such that the image of the natural morphism $\psi \colon V_T \to X$ is dense. Let $\cO_\pi(1)$ be a relatively sufficiently ample divisor class, and $\gamma$ a curve class in fibers given by the complete intersection of $\cO_\pi(1)$ in fibers. For any vector bundle $E$ on $X$, and assume $(\psi^\ast E) \vert_{V_t}$ is $\mu$-stable (resp. $\mu$-semistable) for all $t \in T$ with respect to $\cO_\pi(1)$, then $E$ is stable (resp. semistable) on $X$ with respect to $\psi_\ast \gamma$.
\end{corollary}

\begin{proof}
	Assume the relative dimension of $\pi$ is $d$. Since $\cO_\pi(1)$ is relatively sufficiently ample, whose $(d-1)$-fold self-intersection in each fiber $V_t$ gives the curve class $\gamma$, a general curve $C$ in the class $\gamma$ is smooth. By assumption $(\psi^\ast E) \vert_{V_t}$ is stable (resp. semistable) with respect to $\cO_\pi(1)$, hence $(\pi^\ast E) \vert_C$ is also stable (resp. semistable) for a general curve $C$ in the class $\gamma$, by the Theorem of Mehta and Ramanathan \cite[Theorems 7.2.1, 7.2.8]{huy2}. By assumption the image of such curves is dense in $X$, thus we conclude that $E$ is stable (resp. semistable) by Lemma \ref{lem:slope-stable}.
\end{proof}

We study two special types of curve classes in $P_n(A)$.

Consider the curve class 
\begin{equation}
	\label{eqn:gamma}
	\gamma = h_0^0 h_1^1 h_2^2 \cdots h_n^2 / (H^2)^{n-1} = h_1 (\tau^\ast q_2^\ast [\pt]) \cdots (\tau^\ast q_n^\ast [\pt]).
\end{equation}
It is clear that $\gamma$ is an integral fiber class for the projection $Q_1 \coloneqq q_2 \times \cdots \times q_n$. For each closed point $a \coloneqq (a_2, \cdots, a_n) \in A^{n-2}$, the fiber $$F_a \coloneqq Q_1^{-1}(a) \cong \{ (a_0, a_1) \mid a_0 + a_1 = - (a_2 + \cdots + a_n) \}$$ is the graph of an automorphism of $A$ given as a composition of the antipodal morphism and a translation, hence the projections to both factors $q'_0, q'_1 \colon F_a \to A$ are isomorphisms. Let $i_a \colon F_a \to P_n(A)$ be the inclusion of the fiber, then we have $\gamma = (i_a)_\ast (q'_1)^\ast H$, where we regard $(q'_1)^\ast H$ as a curve class on $F_a$.

\begin{lemma}
	\label{lem:gamma-class}
	For any $\mu_H$-stable (resp. $\mu_H$-semistable) vector bundle $E$ on $A$, $E_i$ is stable (resp. semistable) with respect to $\gamma$ for $i = 0$ or $1$, and semistable with respect to $\gamma$ for each $2 \leqslant i \leqslant n$.
\end{lemma}

\begin{proof}
	If $i = 0$ or $1$, since $q'_i = q_i \circ \tau \circ i_a$, we have $E_i \vert_{F_a} = (q'_i)^\ast E$. Since $E$ is $\mu_H$-stable (resp. $\mu_H$-semistable) and $q'_i$ is an isomorphism, $E_i \vert_{F_a}$ is stable (resp. semistable) with respect to the curve class $(q'_i)^\ast H$. It follows by Corollary \ref{cor:high-dim-fiber} that $E_i$ is stable (resp. semistable) with respect to $\gamma$.
	
	If $i \geqslant 2$, then $E_i \vert_{F_a}$ is a trivial bundle on $F_a$, clearly $\mu_H$-semistable. It follows by Corollary \ref{cor:high-dim-fiber} that $E_i$ is semistable with respect to $\gamma$.
\end{proof}

Consider the curve class 
\begin{equation}
	\label{eqn:delta}
	\delta = h_0h_1h_2 h_3^2 \cdots h_n^2 / (H^2)^{n-2} = h_0 h_1 h_2 (\tau^\ast q_3^\ast [\pt]) \cdots (\tau^\ast q_3^\ast [\pt]).
\end{equation}
It is clear that $\delta$ is an integral fiber class for the projection $Q_2 \coloneqq q_3 \times \cdots \times q_n$. Let $S: A^3 \rightarrow A$ be the addition with respect to the group law on $A$, then for each closed point $a' \coloneqq (a_3, \cdots, a_n) \in A^{n-3}$, the fiber $$ G_{a'} \coloneqq Q_2^{-1}(a') \cong S^{-1}(a_3 + \cdots + a_n). $$ Both of the following lemmas will be useful

\begin{lemma}
	\label{lem:Bertini}
	Let $S: A^3 \rightarrow A$ be the addition with respect to the group law on $A$. For any fixed point $b \in A$, let $q'_i: S^{-1}(b) \rightarrow A$ be the composition of the embedding and the projection to the $i$-th factor for $0 \leqslant i \leqslant 2$. Assume $H$ is a sufficiently ample divisor on $A$. Let $C_i \in \lvert H \rvert$ for $0 \leqslant i \leqslant 2$. Then the following assertions hold:
	\begin{enumerate}
		\item The scheme theoretic intersection $(q'_0)^{-1}(C_0) \cap (q'_1)^{-1}(C_1) \cap (q'_2)^{-1}(C_2)$ in $S^{-1}(b)$ is a smooth curve $C$ for a generic choice of $C_i$ for $0 \leqslant i \leqslant 2$;
		\item Each projection $q'_i: C \rightarrow C_i$ is a finite morphism for $0 \leqslant i \leqslant 2$.
	\end{enumerate}
\end{lemma}

\begin{proof}
	For the first part we use Bertini theorem; see e.g. \cite[Proposition 0.5]{EH16}. We first observe that the addition morphism $S$ is smooth, hence $Y_0 = S^{-1}(b)$ is smooth and irreducible. By assumption the complete linear system $\lvert H \rvert$ has no base point, so does $(q'_0)^{-1}(H)|_{Y_0}$. Hence generic choices of $C_0 \in \lvert H \rvert$ and $Y_1 = Y_0 \cap (q'_0)^{-1}(C_0)$ are smooth and irreducible by Bertini theorem. For the same reason $(q'_1)^{-1}(H)|_{Y_1}$ has no base point, and hence generic choices of $C_1 \in \lvert H \rvert$ and $Y_2 = Y_1 \cap (q'_1)^{-1}(C_1)$ are smooth and irreducible. Similarly, generic choices of $C_2 \in \lvert H \rvert$ and $C = Y_2 \cap (q'_2)^{-1}(C_2)$ are smooth and irreducible. A dimension count shows that $C$ is a curve.
	
	For the second part, since $C$ is smooth irreducible, and the projection $q'_i \colon C \to C_i$ is surjective, \cite[Proposition II.6.8]{hart} implies that it is a finite morphism. 
\end{proof}

\begin{lemma}
	\label{lem:delta-class}
	For any $\mu_H$-stable (resp. $\mu_H$-semistable) vector bundle $E$ on $A$, $E_i$ is stable (resp. semistable) with respect to $\delta$ for each $0 \leqslant i \leqslant 2$, and semistable with respect to $\delta$ for each $3 \leqslant i \leqslant n$.
\end{lemma}

\begin{proof}
	If $0 \leqslant i \leqslant 2$, then for each $a' \in A^{n-3}$, the fiber $G_{a'}$ contains all curves of the form $(q'_0)^{-1}(C_0) \cap (q'_1)^{-1}(C_1) \cap (q'_2)^{-1}(C_2)$, which are in the class $\delta$. By Lemma \ref{lem:Bertini}, a generic member of them is smooth, and the projection gives a finite morphism $q'_i \colon C \to C_i$ such that $E_i \vert_C = (q'_i)^\ast (E \vert_{C_i})$. By the Theorem of Mehta and Ramanathan \cite[Theorems 7.2.1, 7.2.8]{huy2}, $E \vert_{C_i}$ is stable (semistable) for a generic choice of $C_i$. It follows by \cite[Lemmas 3.2.2, 3.2.3]{HL} that $E_i|_C$ is stable (semistable). All such curves $C$ cover a dense subset of $G_{a'}$. Since $a' \in A^{n-3}$ is arbitrary, we conclude that $E_i$ is stable (semistable) with respect to $\delta$ by Lemma \ref{lem:slope-stable}.
	
	If $i \geqslant 2$, then $E_i \vert_{G_{a'}}$ is a trivial bundle on $G_{a'}$, clearly $\mu_H$-semistable. It follows by Corollary \ref{cor:high-dim-fiber} that $E_i$ is semistable with respect to $\delta$.
\end{proof}

It was proven in \cite[Proposition 4.7]{stapleton} that the pullback of a slope stable bundle $E$ from $A$ to $A^{n+1}$ via the projection to any factor is stable with respect to a natural $\mathfrak{S}_{n+1}$-invariant ample class. The following proposition shows that a further restriction to $P_n(A)$, the zero fiber of the addition morphism, remains stable.

\begin{proposition}
	\label{prop:StableE}
	Under the above notations, let $E$ be a $\mu_H$-stable bundle on $A$. Then $E_i$ is a $\mu_{H_P}$-stable bundle on $P_n(A)$ for each $0 \leqslant i \leqslant n$.
\end{proposition}

\begin{proof}
	By replacing $H$ with a high tensor power of itself, we can assume that a generic element $C \in \lvert H \rvert$ in the linear system is a smooth curve such that $E|_C$ is slope stable. We need to show that $E_i$ is slope stable with respect to the curve class $H_P^{2n-1}$. We expand the product to obtain
	\begin{equation*}
		\label{eqn:AmpleExp}
		H_P^{2n-1} = \sum_{\substack{k_0 + \cdots + k_n = 2n-1 \\ 0 \leqslant k_0, \cdots, k_n \leqslant 2}} c_{k_0 \cdots k_n} h_0^{k_0} \cdots h_n^{k_n}
	\end{equation*}
	where each $c_{k_0 \cdots k_n}$ is some positive integer, and each term in the summation is a positive multiple of a curve class of the form \eqref{eqn:gamma} or \eqref{eqn:delta}, possibly with the indices permuted. Lemmas \ref{lem:gamma-class} and \ref{lem:delta-class} guarantee that $E_i$ is semistable with respect to each such curve class, and stable with respect to at least one of these classes. It follows by Lemma \ref{lem:additive-stable} that $E_i$ is stable with respect to the curve class $H_P^{2n-1}$; in other words, $E_0$ is $\mu_{H_P}$-stable.
\end{proof}

\subsection{Tautological bundles}

For any torsion free coherent sheaf $F$ on $\KA$, we follow \cite[\S 1]{stapleton} to define an $\fS_{n+1}$-invariant coherent sheaf on $P_n(A)$ by
$$ (F)_P = (j_P)_\ast \sigma_\circ^\ast (h_\circ^{-1})^\ast j_K^\ast F, $$
which is reflexive if $F$ itself is reflexive. Moreover, we observe that an analogue of \cite[Lemma 1.2]{stapleton}, namely
\begin{equation}
	\label{eqn:MultiCover}
	(n+1)! \int_{K_n(A)} c_1(F) \cdot (H_K)^{2n-1} = \int_{P_n(A)} c_1((F)_P) \cdot (H_P)^{2n-1}
\end{equation}
holds due to the relevant diagonals having codimension $2$. The following result is an analogue of \cite[Theorem 1.4]{stapleton} in the Kummer case.

\begin{proposition}
	\label{thm:boundary-stable}
	Let $E$ be a $\mu_H$-stable bundle on $A$ not isomorphic to $\cO_A$. Then $E^{(n)}$ is a $\mu_{H_K}$-stable bundle on $\KA$.
\end{proposition}

\begin{proof}
	It suffices to show that every reflexive subsheaf of $E^{(n)}$ of smaller rank has a smaller slope. Let $F$ be such a subsheaf of $E^{(n)}$. Then $(F)_P$ is an $\fS_{n+1}$-invariant reflexive subsheaf of $(E^{(n)})_P$. Using equation \eqref{eqn:MultiCover}, it is enough to prove $\mu_{H_P}((F)_P) < \mu_{H_P}((E^{(n)})_P)$. Let $G$ be a non-zero (not necessarily $\fS_{n+1}$-invariant) $\mu_{H_P}$-stable subsheaf of $(F)_P$ of maximal slope; e.g., we can take $G$ to be the first factor in a Jordan-H\"older filtration of $(F)_P$. A similar argument as in \cite[Lemma 1.1]{stapleton} shows that $(E^{(n)})_P = E_0 \oplus \cdots \oplus E_n$ (both sides are reflexive sheaves and isomorphic on $P_n(A)_\circ$ whose complement is of codimension $2$). Therefore, there exists some $i$ such that the composition of the embedding and projection
	\begin{equation}
		\label{eqn:CompIsom}
		G \hookrightarrow (F)_P \hookrightarrow (E^{(n)})_P \twoheadrightarrow E_i
	\end{equation}
	is non-zero. Since $E_i$ is also $\mu_{H_P}$-stable for each $0 \leqslant i \leqslant n$ by Proposition \ref{prop:StableE}, we must have $\mu_{H_P}(G) \leqslant \mu_{H_P}(E_i)$.
	
	\textsc{Case 1.} If $\mu_{H_P}(G) < \mu_{H_P}(E_i)$, then $$\mu_{H_P}((F)_P) \leqslant \mu_{H_P}(G) < \mu_{H_P}(E_i) = \mu_{H_P}((E^{(n)})_P),$$ hence $(F)_P$ does not destabilize $(E^{(n)})_P$.
	
	\textsc{Case 2.} If $\mu_{H_P}(G) = \mu_{H_P}(E_i)$, then the composition map \eqref{eqn:CompIsom} must be an isomorphism. Since $E \not\cong \cO_A$, we have $E_i \not\cong E_j$ for $i \neq j$: Choose any $k$ different from $i$ and $j$. Then the projection $q'' = (q_0, \cdots, q_{k-1}, q_{k+1}, \cdots, q_n)$ identifies $P_n(A)$ with $A^n$. The pullback of a non-trivial sheaf via projections to two distinct factors are not isomorphic. It follows that the composition
	\begin{equation*}
		G \hookrightarrow (F)_P \hookrightarrow (E^{(n)})_P \twoheadrightarrow E_j
	\end{equation*}
	must be zero for any $j \neq i$. It follows that $G$ is the direct summand $E_i$ of $(E^{(n)})_P$. Since $(F)_P$ is an $\fS_{n+1}$-invariant subsheaf of $(E^{(n)})_P$ containing the direct summand $E_i$, we obtain $(F)_P = (E^{(n)})_P$, which cannot happen since the left-hand side has a smaller rank than the right-hand side. This concludes the proof.
\end{proof}

In the following we will apply the perturbation argument in \cite[\S 4]{stapleton} to show the slope stability of $E^{(n)}$ with respect to an ample divisor on $K_n(A)$, which stays constant when we deform $E$ in its moduli.

\subsection{The family of stable bundles}

In this subsection we study families of stable tautological bundles. We assume that $$v = (v_0, v_1, v_2) \in H^0(A, \ZZ) \oplus \NS(A) \oplus H^4(A, \ZZ)$$ is a Mukai vector satisfying the following condition
\begin{itemize}
	\item [$(\dagger)$] the projective moduli space $M_H(v)$ of $H$-semistable sheaves of class $v$ is non-empty and contains only $\mu_H$-stable locally free sheaves.
\end{itemize}
This condition is easy to achieve: First of all we require $v_0 > 0$. In order for every $\mu_H$-semistable sheaf of class $v$ to be stable, it suffices to require that $v_0$ and $H \cdot v_1$ are coprime. The nonemptiness can be achieved by requiring 
\begin{equation}
	\label{eqn:nonempty}
	\langle v^2 \rangle = v_1^2 - 2v_0v_2 \geqslant 0.
\end{equation}
Finally, the local freeness of all $\mu_H$-stable sheaves holds when $v_2$ takes the largest possible value satisfying \eqref{eqn:nonempty} for any fixed $v_0$ and $v_1$. For instance, if $A$ is an abelian surface with a primitive ample divisor $H$ such that $H^2 = 16$, then the Mukai vector $v = (5, H, 1)$ satisfies the condition $(\dag)$.

Under the condition $(\dagger)$, we have seen by Proposition \ref{thm:boundary-stable} that the tautological bundle $E^{(n)}$ is $\mu_{H_K}$-stable for each $E \in M_H(v)$. However, $H_K$ lies in the boundary of the ample cone of $K_n(A)$. In order to establish the stability of the tautological bundle with respect to some ample class, we need the following result

\begin{proposition}\label{prop:sameH2}
	Under condition $(\dagger)$, there exists an ample class $H' \in \NS(\KA)$ near $H_K$, such that $E^{(n)}$ is $\mu_{H'}$-stable for all $[E] \in M_H(v)$.
\end{proposition}

\begin{proof}
	By replacing $H$ with a high tensor power of itself if necessary, we assume the complete linear system $\lvert H_K \rvert$ defines (the restriction of) the Hilbert-Chow morphism $h: \KA \rightarrow S_n(A)$ as shown in \eqref{eqn:Notations}. We claim that $h$ is semismall. Indeed, for any partition $\xi$ given by $n+1 =1 \cdot n_1 + 2 \cdot n_2 + \cdots + r \cdot n_r$, we consider the locally closed subscheme $Y_\xi \subset S_n(A)$ parametrizing $n_1 + n_2 + \cdots + n_r$ distinct points, among which are $n_i$ points of multiplicity $i$ for $1 \leqslant i \leqslant r$. We have $\dim Y_\xi = 2(n_1 + n_2 + \cdots + n_r) -2$ and $\dim h^{-1}(y) = 0 \cdot n_1 + 1 \cdot n_2 + \cdots + (p-1) \cdot n_p$ for any closed point $y \in Y$. It follows that $\dim Y_\xi + 2 \dim h^{-1}(y) = 2n = \dim \KA$ which implies that $h$ is semismall. Therefore $H_K$ is lef by \cite[Definition 2.1.3]{dCM} and satisfies the hard Lefschetz property by \cite[Theorem 2.3.1]{dCM}. It then follows from \cite[Proposition 4.8]{stapleton} that each $E^{(n)}$ is $\mu_{H'}$-stable with respect to some ample class $H'$ near $H_K$. However, in order to find a single $H'$ that works simultaneously for all $E^{(n)}$, we can apply the entire proof of \cite[Proposition 4.8]{stapleton} except one step; namely, we need to find a convex open set $U$ such that $\alpha \coloneqq H_K^{2n-1}$ is in the closure of $U$, and for each $\gamma \in U$, the tautological bundle $E^{(n)}$ is stable with respect to $\gamma$ for all $[E] \in M_H(v)$.
	
	We follow the notations in \cite[Definition 3.1]{greb16}. For each $[E] \in M_H(v)$, $\SStab(E^{(n)})$ is a convex closed set containing $\alpha$. Hence the intersection 
	$$ \overline{U} \coloneqq \bigcap_{[E] \in M_H(v)} \SStab(E^{(n)}) $$
	is also a convex closed set containing $\alpha$. 
	
	We claim that \cite[Theorem 3.4]{greb16} holds for all $E^{(n)}$ simultaneously; namely, we will show that for any $\beta \in \mathrm{Mov}(\KA)^\circ$ (see \cite[Definition 2.1]{greb16} for the notation), there exists some $e \in \mathbb{Q}^+$, such that $(\alpha + \varepsilon \beta) \in \bigcap_{[E] \in M_H(v)} \Stab(E^{(n)})$ for any real $\varepsilon \in [0,e]$.
	
	To prove the claim, we first note by \cite[p.87, Lemma 5(v)]{friedman} that $E$ is $\mu_H$-stable of class $v = (v_0, v_1, v_2)$ if and only if $E^\vee$ is $\mu_H$-stable of class $v^\vee = (v_0, -v_1, v_2)$. Since $\mu_H$-stable sheaves of class $v^\vee$ are bounded, there exists some positive integer $m$ such that $E^\vee(mH)$ is globally generated and $H^i(A, E^\vee(mH))=0$ for all $i>0$. Hence there exists a surjective map $\cO_A(-mH)^{\oplus N} \twoheadrightarrow E^\vee$, where $m$ and $N$ are independent of $E$. Since $E$ is locally free, we can take the dual of the above surjective map to obtain an injective map $E \hookrightarrow \cO_A(mH)^{\oplus N}$, and complete it to an exact sequence
	$$ 0 \longrightarrow E \longrightarrow \cO_A(mH)^{\oplus N} \longrightarrow Q_E \longrightarrow 0. $$
	We apply the functor $q_\ast \circ p^\ast$ on the above sequence. It was proven in \cite[Lemma 3.1]{rz} that the morphism $p$ is flat for $n \geqslant 2$, hence the functor $p^\ast$ is exact. The morphism $q$ is finite, and thus $q_\ast$ is also exact. Therefore we obtain an exact sequence 
	\begin{equation}
		\label{eqn:subsheafE}
		0 \longrightarrow E^{(n)} \longrightarrow q_\ast p^\ast \cO_A(mH)^{\oplus N} \longrightarrow q_\ast p^\ast Q_E \longrightarrow 0.
	\end{equation}
	We note that the slope $c \coloneqq \mu_\beta(E^{(n)})$ is independent of the choice of $[E] \in M_H(v)$, and redefine the set $S$ in the proof of \cite[Theorem 3.4]{greb16} to be 
	$$ S \coloneqq \{ c_1(F) \mid F \subseteq E^{(n)} \text{ for some } [E] \in M_H(v) \text{ such that } \mu_\beta(F) \geqslant c \}. $$
	By \eqref{eqn:subsheafE} we see that $S$ is a subset of
	$$ T \coloneqq \{ c_1(F) \mid F \subseteq q_\ast p^\ast \cO_A(mH)^{\oplus N} \text{ such that } \mu_\beta(F) \geqslant c \}, $$
	which is finite by \cite[Theorem 2.29]{greb16}. Thus $S$ is a also a finite set. We can then apply the rest of the proof of \cite[Theorem 3.4]{greb16} literally to conclude the claim.
	
	Now we can show that $\overline{U}$ is of full dimension $r := \rk N_1(\KA)$. If not, then we have $\alpha \in \overline{U} \subseteq L$ for some hyperplane $L \subset N_1(\KA)_{\mathbb{R}}$. Since $\mathrm{Mov}(\KA)$ is of full dimension, we can choose some $\beta \in \mathrm{Mov}(\KA)^\circ \setminus L$. It follows that $(\alpha + \varepsilon \beta) \in \overline{U} \setminus L$ for some small $\varepsilon > 0$ by the above claim and the choice of $\beta$. This yields a contradiction. 
	
	We define $U$ to be the interior of $\overline{U}$ and claim that $U$ is non-empty. Indeed, since $\overline{U}$ is of full dimension $r$, we can choose $r+1$ points of $\overline{U}$ in general positions, which form an $r$-simplex. By the convexity of $\overline{U}$, the entire simplex is in $\overline{U}$ hence any interior point of the simplex is also an interior point of $\overline{U}$. The convexity of $U$ follows from the convexity of $\overline{U}$. And it is clear from the construction that $\alpha = H_K^{2n-1}$ is in the closure of $U$. 
	
	We finally prove that $U \subseteq \bigcap_{[E] \in M_H(v)} \Stab(E^{(n)})$. If not, suppose that there is some $\gamma_0 \in U$ and some $[E] \in M_H(v)$, such that $\gamma_0 \in \SStab(E^{(n)}) \setminus \Stab(E^{(n)})$; namely, $\mu_{\gamma_0}(F) = \mu_{\gamma_0}(E^{(n)})$ for some proper subsheaf $F$ of $E^{(n)}$. Since the slope function is linear with respect to the curve class, and $\mu_\alpha(F) < \mu_\alpha(E^{(n)})$ by Proposition \ref{prop:StableE}, one can find a hyperplane in $N^1(\KA)_\mathbb{R}$ through $\gamma_0$, such that $\mu_\gamma(E^{(n)}) - \mu_\gamma(F)$ takes opposite signs for $\gamma$ in the two open half-spaces separated by the hyperplane. In particular, $F$ destabilizes $\mathcal{U}_{\da_0}$ in one of the half-spaces. Since $U$ has non-empty intersection with both half-spaces, this contradicts the condition $U \subseteq \SStab(E^{(n)})$. Therefore we have $U \subseteq \bigcap_{[E] \in M_H(v)} \Stab(E^{(n)})$, as desired.
\end{proof}

\subsection{A component of the moduli space}

In this subsection, we show that under some favorable numerical conditions, $M_H(v)$ is isomorphic to a connected component of a moduli space of stable sheaves on $K_n(A)$. 

Indeed, we still assume that $v$ satisfies condition $(\dagger)$; or more precisely, the numerical conditions in the paragraph below $(\dagger)$ that ensure its validity. We further assume the following condition:
\begin{itemize}
	\item [$(\ddagger)$] for every $[E] \in M_H(v)$, we have $H^i(A, E)=0$ for $i>0$.
\end{itemize}
This condition is also easy to achieve. Since all stable sheaves are bounded, there exists some positive integer $m$ independent of the choice of $E$, such that $H^i(A, E(mH))=0$ for all $i>0$. By replacing $v$ with $v \cdot \ch(mH)$, we obtain a Mukai vector $v$ satisfying both $(\dagger)$ and $(\ddagger)$.

Under the above assumptions, let $H'$ be the ample line bundle constructed in Proposition \ref{prop:sameH2}, and let $\cM$ the moduli space of $\mu_{H'}$-stable sheaves on $K_n(A)$ with the same numerical invariants as $E^{(n)}$. By applying Proposition \ref{prop:sameH2}, the integral functor $q_\ast \circ p^\ast$ induces a morphism
\begin{equation}
	\label{eqn:classify1}
	f: M_H(v) \longrightarrow \cM, \quad [E] \longmapsto [E^{(n)}].
\end{equation}
In fact the morphism $f$ can be described as follows:

\begin{theorem}\label{thm:firstmain}
	Under the assumptions $(\dag)$ and $(\ddag)$, the classifying morphism \eqref{eqn:classify1} identifies $M_H(v)$ with a smooth connected component of $\cM$.
\end{theorem}

\begin{proof}
	By \cite[Lemma 1.6.]{rz} we have to prove that $f$ is injective on closed points and that $\dim(T_{[E^{(n)}]}\mathcal{M})= \dim (T_{[E]} M_H(v))$ for all $[E] \in M_H(v)$.
	
	The main tool for achieving this is \cite[Theorem 6.9]{meac}, which is a formula for computing various extension groups between tautological sheaves. More exactly \cite[Theorem 6.9]{meac} implies for any $[E_1], [E_2] \in M_H(v)$ that
	$$ \Hom_{K_n(A)}(E_1^{(n)}, E_2^{(n)}) \cong \Hom_A(E_1, E_2) = \begin{cases}
		\CC, & \text{when } E_1 \cong E_2; \\
		0, & \text{when } E_1 \not\cong E_2.
	\end{cases} $$
	In particular, the case of $E_1 \not\cong E_2$ implies that $f$ is injective on closed points. Moreover, \cite[Theorem 6.9]{meac} also implies
	\begin{align*}
		&\phantom{=\ } \Ext^1_{K_n(A)}(E_1^{(n)}, E_2^{(n)}) \\
		&\cong \Ext^1_A(E_1,E_2) \oplus H^1(A, E_1^\vee) \otimes H^0(A, E_2) \oplus H^0(A, E_1^\vee) \otimes H^1(A, E_2) \\
		&= \Ext^1_A(E_1, E_2),
	\end{align*}
	where the last equality follows from $(\ddag)$ and the Serre duality on $A$. In particular, when $[E_1]$ and $[E_2]$ represent the same closed point $[E] \in M_H(v)$, we obtain that $\dim T_{E^{(n)}}(\cM) = \dim T_E (M_H(v))$ as desired.
\end{proof}

\section{Universal Bundles}
 
In this section we want to construct a second type of stable bundles on $\KA$. The basic idea is to use the Fourier-Mukai transform to find a fine moduli space of stable sheaves $M_{\widehat{H}}(w)$ on $\dA$ such that the generalized Kummer $K_{\widehat{H}}(w)$ in $M_{\widehat{H}}(w)$ is isomorphic to $\KA$. We restrict the universal family of $M_{\widehat{H}}(w)$ to $K_{\widehat{H}}(w)$ and study its fibers over a point $\da\in \dA$, which is a sheaf on $K_{\widehat{H}}(w)\cong \KA$.

\subsection{Stable sheaves on abelian surfaces}
Pick $n,r\in\mathbb{N}$ with $n\geqslant 2$ as well as $r\geqslant n+2$ and let $A$ be an abelian surface satisfying
\begin{equation*}
\NS(A)=\mathbb{Z}H\,\,\,\text{such that}\,\,\,H^2=2(n+r+1).
\end{equation*}
We denote the dual abelian surface by $\dA$. We have the Poincar\'{e} line bundle $\Pb$ on $A\times \dA$ which defines the classical Fourier-Mukai transform
\begin{equation*}
	\Phi: \Db(A) \rightarrow \Db(\dA),\,\,\, E\mapsto Rp_{*}(\Pb\otimes q^{*}(E))
\end{equation*}
where $p: A\times \dA \rightarrow \dA$ and $q: A\times \dA\rightarrow A$ are the projections.

\begin{remark}
Recall that an element $E\in\Db(A)$ is said to be $\mathrm{WIT}_i$ with respect to $\Phi$ if there is a coherent sheaf $G$ on $\dA$ such that $\Phi(E)\cong G[-i]$ in $\Db(\dA)$, where $G[-i]$ is the associated complex concentrated in degree $i$. We say that $E$ is $\mathrm{IT}_i$ if in addition $G$ is locally free, see \cite[Definition 2.3]{muk2}
\end{remark}

Using the canonical isomorphism $A\cong \ddA$ (given by the Poincar\'{e} bundle), we can also understand $\Pb$ as the Poincar\'{e} bundle on $\dA\times A$ up to switching the factors, see \cite[p.198, Remark 9.12]{huy}. This gives rise to the Fourier-Mukai transform
\begin{equation*}
	\widehat{\Phi}: \Db(\dA) \rightarrow \Db(A),\,\,\, F\mapsto Rq_{*}(\Pb\otimes p^{*}(F))
\end{equation*}
It is well known that $\det(\Phi(\mathcal{O}_A(H)))^{-1}$ defines the canonical polarization $\widehat{H}$ on $\dA$ and $\NS(\dA)=\mathbb{Z}\widehat{H}$, see for example \cite{lange}.
\medskip

Now we look at the Mukai vector
\begin{equation*}
	v=(1,H,r)
\end{equation*}
and denote the moduli space of $\mu_H$-semistable sheaves on $A$ with Mukai vector $v$ by $M_H(v)$. Then there is an isomorphism 
\begin{equation*}
\epsilon: \An\times \dA\stackrel{\cong}{\longrightarrow} M_H(v),\,\,\, (Z,\da)\longmapsto I_Z(H)\otimes \Pb_{\da}.
\end{equation*}

We compute $\langle v^2 \rangle=H^2-2r=2(n+r+1)-2r=2n+2$ and thus 
\begin{equation}\label{dim}
\dim(M_H(v))=2n+4.
\end{equation}
Furthermore by the choice of $r$ we have $r>\frac{\langle v^2 \rangle}{2}$, which by \cite[Corollary 3.3]{yosh} implies that every $E\in M_H(v)$ is $\mathrm{IT}_0$ with respect to $\Phi$ and that $\Phi(E)$ is a $\mu_{\widehat{H}}$-stable locally free sheaf on $\dA$ with Mukai vector
\begin{equation*}
	w=(r,-\widehat{H},1).
\end{equation*}
By \cite[Prop. 3.2, Cor. 3.3]{yosh} we get that the Fourier-Mukai transform induces an isomorphism
\begin{equation*}
	\Phi: M_H(v)\stackrel{\cong}{\longrightarrow} M_{\widehat{H}}(w).
\end{equation*}

\begin{remark}\label{fine}
The moduli space $M_{\widehat{H}}(w)$ is fine as $\mathrm{gcd}(r,\widehat{H}^2,1)=1$. Furthermore \cite[Corollary 3.3]{yosh} also shows that all sheaves classified by $M_{\widehat{H}}(w)$ are $\mu_{\widehat{H}}$-stable locally free sheaves.
\end{remark}

\subsection{Generalized Kummer varieties}
We recall the original construction of the generalized Kummer variety due to Beauville, see \cite[Sect. 7]{beau}: the group law on $A$ defines, via the symmetric power and the Hilbert-Chow morphism, a summation morphism:
\begin{equation*}
	\Sigma: \An \rightarrow A^{(n+1)}\rightarrow  A.
\end{equation*}
The generalized Kummer variety is then defined by $\KA:=\Sigma^{-1}(0_A)$.
\medskip

This construction was generalized by Yoshioka to moduli spaces of stable sheaves $M_H(v)$ on $A$, see \cite[Theorem 4.1., Definition 4.1.]{yosh}. We quickly summarize his main results: let $v$ be a primitive Mukai vector with $\langle v^2 \rangle+2\geqslant 6$ and $H$ be a generic polarization, i.e. $\overline{M_H(v)}=M_H(v)$. One finds that the Albanese morphism of $M_H(v)$ is given by 
\begin{equation*}
	\mathfrak{a}_v: M_H(v) \rightarrow A\times \dA
\end{equation*}
with
\begin{equation*}
	\mathfrak{a}_v(E)=\left(\det(\Phi(E))\otimes\det(\Phi(E_0))^{-1},\det(E)\otimes\det(E_0)^{-1} \right) 
\end{equation*}
for some fixed $E_0\in M_H(v)$. Then one can give the following definition:
\begin{definition}
	The generalized Kummer variety $K_H(v)$ in $M_H(v)$ is defined to be the fiber of $\mathfrak{a}_v$ over the point $(0_A,0_{\dA})$, i.e. $K_H(v)=\mathfrak{a}_v^{-1}((0_A,0_{\dA}))$.
\end{definition}
Note that we have $\dim(K_H(v))=2n$ by \eqref{dim}. Now assume that $v$ also satisfies all conditions from \cite[Corollary 3.3]{yosh}, that is the Fourier-Mukai transform induces an isomorphism $\Phi : M_H(v) \stackrel{\cong}{\longrightarrow} M_{\widehat{H}}(w)$. Under these circumstances not only are the moduli spaces isomorphic, but also the induced generalized Kummer varieties, as seen in the following lemma: 
\begin{lemma}
	The isomorphism $\Phi: M_H(v)\stackrel{\cong}{\longrightarrow} M_{\widehat{H}}(w)$ restricts to an isomorphism between generalized Kummer varieties $K_H(v)\stackrel{\cong}{\longrightarrow} K_{\widehat{H}}(w)$.
\end{lemma}
\begin{proof}
We first note that the Albanese morphism $\mathfrak{a}_w: M_{\widehat{H}}(w) \rightarrow \dA \times \ddA$ can be understood as a morphism $\mathfrak{a}_w: M_{\widehat{H}}(w) \rightarrow A \times \dA$ after identifying $A\cong \ddA$ and switching the factors. It is then given by
\begin{equation*}
	\mathfrak{a}_w(F)=(\det(F)\otimes\det(F_0)^{-1},\det(\widehat{\Phi}(F))\otimes\det(\widehat{\Phi}(F_0))^{-1})
\end{equation*}
with $F_0=\Phi(E_0)\in M_{\widehat{H}}(w)$.
\medskip

Using the isomorphism $\varphi: A\times \dA \rightarrow A\times \dA$ given by $\varphi \coloneqq 1_A\times (-1_A)^{*}$ we claim that the following diagram commutes:
\begin{equation*}
	\begin{tikzcd}
		M_H(v) \arrow[r,"\Phi"]\arrow[d, swap, "\mathfrak{a}_v"] & M_{\widehat{H}}(w)\arrow[d, swap, "\mathfrak{a}_w"]\\
		A\times \dA \arrow[r,"\varphi"] & A\times \dA.
	\end{tikzcd}
\end{equation*}
To see this we simply note that since every $E\in M_H(v)$ is $\mathrm{IT}_0$ with respect to $\Phi$, we have the following isomorphism by \cite[Corollary 2.4.]{muk2}:
\begin{equation*}
 \widehat{\Phi}(\Phi(E))\cong(-1_A)^{*}E.
\end{equation*}
Using $\varphi((0_A,0_{\dA}))=(0_A,0_{\dA})$ and the commutativity, we see that $\Phi$ restricts to an isomorphism $K_H(v)\cong K_{\widehat{H}}(w)$.
\end{proof}

\subsection{Construction of a universal family}
In this section we want to construct a universal family for the generalized Kummer variety $K_{\widehat{H}}(w)$. For this we first note that $M_H(v)$ is a fine moduli space, that is there is a universal family on the product $A\times M_H(v)$. Denote the restriction of the universal family along the closed immersion $A\times K_H(v) \hookrightarrow A\times M_H(v)$ by $\mathcal{E}$.

\begin{remark}
For the Mukai vector $v=(1,H,r)$ we have the following isomorphism:
\begin{equation*}
	\KA \stackrel{\cong}{\longrightarrow} K_H(v),\,\,\, [Z]\longmapsto I_Z(H).
\end{equation*}
By making a careful choice of the line bundles on $A$ and $\dA$ representing $\det(E_0)$ and $\det(\Phi(E_0))$ as in \cite[\S 3.1]{gul}, an explicit computation similar to \cite[Lemma 3.2]{gul} shows that there is a commutative diagram
\begin{equation*}
	\begin{tikzcd}
		\An\times \dA \arrow[r,"\epsilon"]\arrow[d, swap, "\Sigma\times 1_{\dA}"] & M_{H}(v)\arrow[d, "\mathfrak{a}_v"]\\
		A\times \dA \arrow[r,"\rho"] & A\times \dA 
	\end{tikzcd}
\end{equation*}
with the isomorphism
\begin{equation*}
    \rho: A\times \dA \stackrel{\cong}{\longrightarrow} A\times\dA,\,\,\, (a,\da)\longmapsto (-a+\phi_{\widehat{H}^{-1}}(\da),\da).
\end{equation*}
Again, as $\rho(0_A,0_{\dA})=(0_A,0_{\dA})$, we find that the isomorphism $\epsilon$ restricts to an isomorphism between the fibers of $\Sigma\times 1_{\dA}$ and $\mathfrak{a}_v$ over $(0_A,0_{\dA})$. It remains to note that these fibers are $K_n(A)$ and $K_H(v)$ by definition.
\end{remark}

Using the last remark we will, from now on, understand the universal family $\mathcal{E}$ on $A\times K_H(v)$ as a family on $A\times \KA$, which is easily seen to be given by
\begin{equation*}
	\mathcal{E}=\IZ\otimes \pi_1^{*}\mathcal{O}_A(H)
\end{equation*}
where $\pi_1: A\times \KA \rightarrow A$ is the projection and $\IZ$ is the universal ideal sheaf on $A\times\KA$.
\medskip

We now define a family $\mathcal{U}$ on $\dA\times\KA$ using the Fourier-Mukai transform relative to $\KA$ following \cite[Sect. 1]{muk}. For this we introduce some notation:
\begin{equation*}
	\begin{tikzcd}[column sep = large, row sep = large]
		& & A\times \dA\times \KA \arrow[ld,swap,"p_{A}"]\arrow[d,"q"]\arrow[rd, "p_{\dA}"] & \\ A &  A\times \KA \arrow[l,swap,"\pi_1"]  & A\times \dA & \dA\times \KA 		
	\end{tikzcd}
\end{equation*}
Then the relative Fourier-Mukai transform is defined by:
\begin{equation*}
\Psi: \Db(A\times \KA)\rightarrow \Db(\dA\times \KA),\,\,\,\, \mathcal{F}\mapsto R{p_{\dA}}_{*}(q^{*}\Pb\otimes p_A^{*}(\mathcal{F}))
\end{equation*}
Using this we define the following family on $\dA\times \KA$:
\begin{equation*}
	\mathcal{U}:=\Psi(\mathcal{E}).
\end{equation*}

The restriction of $\mathcal{E}$ to the fiber over $[Z]\in \KA$ is just $I_Z(H)$ which is $\mathrm{IT}_0$ with respect to $\Phi$, implying that $\mathcal{U}$ is $\mathrm{WIT}_0$ and that $\Psi(\mathcal{E})$ commutes with an arbitrary base change $T\rightarrow \KA$ by \cite[Theorem 1.6.]{muk}. By choosing $T=\{[Z]\}$ for some $[Z]\in \KA$ we see that there is an isomorphism
\begin{equation*}
	\mathcal{U}\otimes\mathcal{O}_{[Z]}=\Psi(\mathcal{E})\otimes \mathcal{O}_{[Z]} \cong \Phi(\mathcal{E}\otimes \mathcal{O}_{[Z]}) \cong \Phi(I_Z(H)).
\end{equation*}
As $\Phi(I_Z(H))$ is locally free the last equation shows that $\mathcal{U}$ is locally free by \cite[Lemma 2.1.7]{HL}.
\medskip

Furthermore, since $H^1(A,I_Z(H))=0$ for all $[Z]\in K_n(A)$ we see, using standard results from the theory of cohomology and base change, that for every morphism $\alpha: S \rightarrow \dA\times K_n(A)$ we get a diagram
\begin{equation}\label{basech}
	\begin{tikzcd}
		A & A\times S \arrow[l,swap ,"t_1"] \arrow[r,"\beta"]\arrow[d, swap, "t_2"] & A\times \dA\times \KA\arrow[d, "p_{\dA}"]\arrow[r, "p_A"] & A\times \KA\\
		& S \arrow[r,"\alpha"] & \dA\times \KA & 
	\end{tikzcd}
\end{equation}
together with an isomorphism:
\begin{align*}
	\alpha^{*}\mathcal{U}&=\alpha^{*}(R{p_{\dA}}_{*}(q^{*}\Pb\otimes p_A^{*}\mathcal{E}))\\
	&\cong R{t_2}_{*}\beta^{*}(q^{*}\Pb\otimes p_A^{*}\mathcal{E}).
\end{align*}
We sum up the results from this subsection in the following lemma:
\begin{lemma}
The family $\mathcal{U}=\Psi(\mathcal{E})$ on $\dA\times K_n(A)$ is a locally free universal family for $K_{\widehat{H}}(w)$, namely, its classifying morphism $K_n(A) \rightarrow M_{\widehat{H}}(w)$ induces the isomorphism 
\begin{equation*}
  K_n(A)\stackrel{\cong}{\longrightarrow} K_{\widehat{H}}(w),\,\,\,[Z]\longmapsto \Phi(I_Z(H)).  
\end{equation*}
\end{lemma}

\subsection{Stability of the wrong-way fibers}
In this section we want to study the stability of the wrong-way fibers of $\mathcal{U}$, that is the fibers over points  $\da\in\dA$. For this we choose in the diagram \eqref{basech} the base change along the inclusion $j_{\da}$ of the fiber over $\da$ of the projection $\dA\times \KA\rightarrow \dA$, that is
\begin{equation}
	\begin{tikzcd}
		A & A\times \KA \arrow[l,swap ,"t_1"] \arrow[r,"i_{\da}"]\arrow[d, swap, "t_2"] & A\times \dA\times \KA\arrow[d, "p_{\dA}"]\arrow[r, "p_A"] & A\times \KA\\
		& \KA \arrow[r,"j_{\da}"] & \dA\times \KA & 
	\end{tikzcd}
\end{equation}
where the morphisms $j_{\da}$ and $i_{\da}$ are given on closed points by 
\begin{align*}
	j_{\da}: \KA &\hookrightarrow \dA\times \KA,\,\,\,\, [Z]\mapsto (\da,[Z])\\
	i_{\da}: A\times \KA &\hookrightarrow A\times \dA\times \KA,\,\,\,\, (a,[Z])\mapsto (a,\da,[Z]).
\end{align*}
Going through the base change we see that we can describe the wrong-way fibers in the following way:
\begin{align*}
	\mathcal{U}_{\da}=j_{\da}^{*}\mathcal{U}&=j_{\da}^{*}(R{p_{\dA}}_{*}(q^{*}\Pb\otimes p_A^{*}\mathcal{E}))\\
&\cong R{t_2}_{*}i_{\da}^{*}(q^{*}\Pb\otimes p_A^{*}(\IZ\otimes \pi_1^{*}\mathcal{O}_A(H)))\\
&\cong R{t_2}_{*}(\IZ\otimes t_1^{*}(\Pb_{\da}(H))).
\end{align*}

We recall the integral functor
\begin{equation}\label{theta}
	\Theta: \Db(A)\rightarrow \Db(\KA),\,\,\,E\mapsto R{t_2}_{*}(\IZ\otimes t_1^{*}E),
\end{equation}
which is a $\mathbb{P}^{n-1}$-functor by \cite[Theorem 4.1]{meac}.
\medskip

We see that the wrong-way fiber is given by 
\begin{equation}\label{eqn:U}
\mathcal{U}_{\da}=\Theta(\Pb_{ \da}(H))
\end{equation}
and sits in the exact sequence:
\begin{equation}\label{exsequ}
	\begin{tikzcd}
		0 \arrow[r] & \mathcal{U}_{\da} \arrow[r] & R{t_2}_{*}(t_1^{*}(\Pb_{\da}(H))) \arrow[r] & R{t_2}_{*}(\mathcal{O}_{\mathcal{Z}}\otimes t_1^{*}(\Pb_{\da}(H))) \arrow[r] & 0
	\end{tikzcd}
\end{equation}
We also have:
\begin{equation*}
	R{t_2}_{*}(t_1^{*}(\Pb_{\da}(H)))\cong H^0(\Pb_{\da}(H))\otimes \mathcal{O}_{\KA}
\end{equation*}
by cohomology and base change. Furthermore
\begin{equation*}
	R{t_2}_{*}(\mathcal{O}_{\mathcal{Z}}\otimes t_1^{*}(\Pb_{\da}(H)))=(\Pb_{\da}(H))^{(n)}
\end{equation*}
is the tautological bundle of rank $n+1$ on $\KA$ induced by $\Pb_{\da}(H)$.

A quick diagram chase shows that we have
\begin{equation*}
(\Pb_{\da}(H))^{(n)}\cong \iota^{*}((\Pb_{\da}(H))^{[n+1]})
\end{equation*}
where $\iota: \KA \hookrightarrow A^{[n+1]}$ is the inclusion and $(\Pb_{\da}(H))^{[n+1]}$ is the tautological bundle induced by $\Pb_{\da}(H)$ on $A^{[n+1]}$.

For the next results we recall that we have $\NS(\KA)=\NS(A)_K\oplus \mathbb{Z}\delta$. Here $D_K$ is the divisor class on $\KA$ induced by the divisor class $D$ on $A$ and $\delta$ is a divisor class on $\KA$ such that $2\delta=[E]$ where $E$ is the exceptional divisor of the Hilbert-Chow morphism $\KA \rightarrow S_n(A)$. In our case this reads
\begin{equation*}
	\NS(\KA)=\ZZ H_K\oplus\ZZ\delta.
\end{equation*}
\begin{remark}
Note that we can also write $\NS(\KA)=\iota^{*}\NS(A^{[n+1]})$, with
\begin{equation*}
\NS(A^{[n+1]})=\NS(A)_{n+1}\oplus \mathbb{Z}\Delta\oplus{\Sigma}^{*}\NS(A).
\end{equation*}
where $\NS(A)_{n+1}$ are the divisor classes on $A^{[n+1]}$ induced from $A$ and $\Delta$ is the class such that $2\Delta$ is the class of the exceptional divisor of $A^{[n+1]}\rightarrow A^{(n+1)}$. We have $\iota^{*}H_{n+1}=H_K$ and $\iota^{*}\Delta=\delta$.
\end{remark}

\begin{lemma}
\label{lem:Chernclass}
We have $c_1(\mathcal{U}_{\da})=-H_K+\delta$.
\end{lemma}

\begin{proof}
By the exact sequence \eqref{exsequ} we get:
\begin{align*}
	c_1(\mathcal{U}_{\da})&=-c_1((\Pb_{\da})^{(n)})\\
	&=-c_1(\iota^{*}((\Pb_{\da}(H))^{[n+1]}))\\
	&=-\iota^{*}c_1((\Pb_{\da}(H))^{[n+1]})\\
	&=-\iota^{*}(H_{n+1}-\Delta)=-H_K+\delta
\end{align*}
where we use \cite[Lemma 1.5]{wandel} in the second to last step.
\end{proof}

To compute slopes on $\KA$ we need the following intersection numbers, which can, for example, be found in \cite[p.9]{britze}:

\begin{lemma}\label{intersect}
	For the classes $H_K$ and $\delta$ from $\NS(\KA)$ we have:
	\begin{itemize}
		\setlength\itemsep{1em}
		\item $H_K^{2n}=\frac{(n+1)(2n)!}{(n)!2^{n}}(H^2)^{n} > 0$
		\item $H_K^{2n-1}\delta=0$.
	\end{itemize}
\end{lemma}

\begin{lemma}\label{equiv-NS}
There is an isomorphism
\begin{equation*}
\NS(A)\stackrel{\cong}{\longrightarrow} \NS(P_n(A))^{\mathfrak{S}_{n+1}},\,\,\,H \longmapsto \sum\limits_{i=0}^n \tau^{*}q_i^\ast H.
\end{equation*}
\end{lemma}
\begin{proof}
	We note that $P_n(A)$ is itself an abelian variety (isomorphic to $A^n$ via projection), and hence its integral cohomology is torsion free. This implies especially that its Neron-Severi group $\NS(P_n(A))$ is torsion free and hence so is $\NS(P_n(A))^{\mathfrak{S}_{n+1}}$.
	
	Furthermore by \cite[Lemma 3]{hashi} we have an isomorphism
	\begin{equation*}
		\NS(P_n(A))^{\mathfrak{S}_{n+1}}\otimes \mathbb{Q}\cong \left( \NS(P_n(A))\otimes\mathbb{Q}\right) ^{\mathfrak{S}_{n+1}}.
	\end{equation*}
 and so it is enough to prove the lemma over the field of rational numbers $\mathbb{Q}$.
 
We start with the morphisms
\begin{equation*}
	\begin{tikzcd}
		 P_n(A) \arrow[hookrightarrow]{r}{\tau} & A^{n+1} \arrow[r, "S"] & A 
	\end{tikzcd}
\end{equation*}
where $S=\sum\limits_{i=0}^n q_i$ is the summation morphism using the group law on $A$.

The natural inclusion $\tau$ has the following retract:
\begin{equation*}
	 A^{n+1} \rightarrow P_n(A),\,\,\,(a_0,\ldots,a_n)\mapsto (a_0,\ldots,a_{n-1},-\sum\limits_{i=0}^{n-1} a_i),
\end{equation*}
which shows that we have a surjection
\begin{equation*}
		\begin{tikzcd}
		H^2(A^{n+1},\mathbb{Q}) \arrow[r, "\tau^{*}"] & H^2(P_n(A),\mathbb{Q}) \arrow[r] & 0 
	\end{tikzcd}
\end{equation*}
As we work over $\mathbb{Q}$ and $\mathfrak{S}_{n+1}$ is finite we get an induced surjection:
\begin{equation*}
	\begin{tikzcd}
		H^2(A^{n+1},\mathbb{Q})^{\mathfrak{S}_{n+1}} \arrow[r, "\tau^{*}"] & H^2(P_n(A),\mathbb{Q})^{\mathfrak{S}_{n+1}} \arrow[r] & 0 .
	\end{tikzcd}
\end{equation*}
It is well known, see for example \cite[Theorem 2.15]{lehn}, that:
\begin{equation*}
	H^2(A^{n+1},\mathbb{Q})^{\mathfrak{S}_{n+1}}\cong H^2(A,\mathbb{Q})\oplus \Lambda^2 \left(H^1(A,\mathbb{Q}) \right) 
\end{equation*}
where the maps are given by:
\begin{equation*}
	H^2(A,\mathbb{Q}) \hookrightarrow H^2(A^{n+1},\mathbb{Q})^{\mathfrak{S}_{n+1}},\,\,\, c \mapsto \sum\limits_{i=0}^n q_i^{*}c 
\end{equation*}
as well as (using $\Lambda^2 \left(H^1(A,\mathbb{Q})\right)\cong H^2(A,\mathbb{Q})$):
\begin{equation*}
	\Lambda^2 \left(H^1(A,\mathbb{Q})\right) \hookrightarrow H^2(A^{n+1},\mathbb{Q})^{\mathfrak{S}_{n+1}},\,\,\, c\wedge d \mapsto  \sum\limits_{i,j} \left(q_i^{*}c \wedge q_j^{*}d\right)  
\end{equation*}
Now since $S=\sum\limits_{i=0}^n q_i$, similarly to Beauville in \cite[Proposition 8.]{beau}, we get:
\begin{equation*}
	\sum\limits_{i,j} \left(q_i^{*}c \wedge q_j^{*}d\right) =\left( \sum\limits_{i=0}^n q_i^{*}c\right) \wedge \left( \sum\limits_{j=0}^n q_j^{*}d\right)=S^{*}(c\wedge d).
\end{equation*}
This implies $\Lambda^2 \left(H^1(A,\mathbb{Q})\right)\cong \Im(S^{*})$. But then
\begin{equation*}
	\tau^{*}(S^{*}(c\wedge d))=(S\circ \tau)^{*}(c\wedge d)=0
\end{equation*}
which shows that we have
\begin{equation}
	H^2(P_n(A),\mathbb{Q})^{\mathfrak{S}_{n+1}}\cong \tau^{*}H^2(A,\mathbb{Q}).
\end{equation}
Using the Lefschetz $(1,1)$-theorem gives
\begin{equation*}
	\left( \NS(P_n(A))\otimes\mathbb{Q}\right)^{\mathfrak{S}_{n+1}}\cong \tau^{*}\left(\NS(A)\otimes \mathbb{Q}\right), 
\end{equation*}
which is what we wanted to prove.
\end{proof}

\begin{proposition}
	\label{prop:stableEx}
	The vector bundle $\mathcal{U}_{\da}$ defined in \eqref{eqn:U} is slope stable with respect to $H_K$.
\end{proposition}

\begin{proof}
	We follow the idea in the proof of \cite[Theorem 1.4]{stapleton}. 
	
	Since $j_K^{*}(-)$, $(h_\circ^{-1})^{*}(-)$ and $\sigma_\circ^\ast(-)$ are exact, and $(j_P)_{*}$ is left exact, by applying these functors to \eqref{exsequ} 
	we obtain an exact sequence of $\mathfrak{S}_{n+1}$-invariant reflexive sheaves on $P_n(A)$ as follows:
	$$ 0 \longrightarrow (\mathcal{U}_{\da})_{P} \longrightarrow (H^0(\Pb_{\da}(H))\otimes \mathcal{O}_{\KA})_{P} \stackrel{\varphi}{\longrightarrow} (\Pb_{\da}(H))^{(n)}_{P} $$
	where $\varphi$ is not necessarily surjective. It is clear that 
	$$ (H^0(\Pb_{\da}(H))\otimes \mathcal{O}_{\KA})_{P}=H^0(\Pb_{\da}(H))\otimes \mathcal{O}_{P_n(A)}, $$
	and we also have
	$$ ((\Pb_{\da}(H))^{(n)})_{P} = \bigoplus\limits_{i=0}^n \tau^{*} q_i^\ast \left( \Pb_{\da}(H)\right)  $$
	by a similar argument as in \cite[Lemma 1.1]{stapleton} (see also Proposition \ref{thm:boundary-stable}). Hence the above sequence becomes
	\begin{equation}
		\label{eqn:equisheaf}
		0 \longrightarrow (\mathcal{U}_{\da})_{P} \longrightarrow H^0(\Pb_{\da}(H))\otimes \mathcal{O}_{P_n(A)} \stackrel{\varphi}{\longrightarrow} \bigoplus\limits_{i=0}^n \tau^{*} q_i^\ast \left( \Pb_{\da}(H)\right)
	\end{equation}
	where $\varphi$ is the evaluation map on $P_n(A)_\circ$. 
	
	More precisely, for any set of closed points $(a_0,\ldots, a_n) \in P_n(A)$ with $a_i \neq a_j$, the morphism of fibers can be identified as
	\begin{align*}
		\varphi_{(a_0,\ldots, a_n)}: H^0(\Pb_{\da}(H)) &\longrightarrow \bigoplus\limits_{i=0}^n \left( \Pb_{\da}(H)\right) _{x_i} \\
		s &\longmapsto (s(a_0),\ldots, s(a_n))
	\end{align*}
	Since for any non-trivial $s \in  H^0(\Pb_{\da}(H))$, there are always (many choices of)  distinct points $(a_0,\ldots a_n) \in P_n(A)$ such that $(s(a_0),\ldots, s(a_n)) \neq (0,\ldots, 0)$, we conclude that the map of global sections
	$$ H^0(\varphi):  H^0(\Pb_{\da}(H)) \longrightarrow H^0(\bigoplus\limits_{i=0}^n \tau^{*} q_i^\ast \Pb_{\da}(H)) $$
	is injective. It follows by \eqref{eqn:equisheaf} that $H^0((\mathcal{U}_{\da})_{P}) = 0$.
	
	Note that $\varphi$ is surjective on $P_n(A)_\circ$, hence $\coker(\varphi)$ is supported on the big diagonal of $P_n(A)$ which is of codimension $2$. It follows that
	$$ c_1((\mathcal{U}_{\da})_{P}) = -\sum\limits_{i=0}^n \tau^{*} q_i^\ast H. $$
	
	We claim that $(\mathcal{U}_{\da})_{P}$ has no $\mathfrak{S}_{n+1}$-invariant subsheaf which is destabilizing with respect to $H_{P}$. Indeed, assume $F$ is an $\mathfrak{S}_{n+1}$-invariant subsheaf of $(\mathcal{U}_{\da})_{P}$, then $c_1(F)\in \NS(P_n(A))^{\mathfrak{S}_{n+1}}$ and thus by Lemma \ref{equiv-NS} we have:
		\begin{equation*}
			c_1(F) = a(\sum\limits_{i=0}^n \tau^{*}q_i^\ast H)\,\,\,\text{for some $a \in \mathbb{Z}$}.
		\end{equation*}

	If $a \leqslant -1$, then 
	\begin{equation*}
		c_1(F) \cdot H_P^{2n-1} \leqslant c_1((\mathcal{U}_{\da})_{P}) \cdot H_P^{2n-1} < 0
	\end{equation*}
	Since $1 \leqslant \rk(F) < \rk((\mathcal{U}_{\da})_{P})$, it follows that $\mu_{H_P}(F) < \mu_{H_P}((\mathcal{U}_{\da})_{P})$, hence $F$ is not destabilizing.
	
	If $a=0$, we choose a (not necessarily $\mathfrak{S}_{n+1}$-invariant) non-zero stable subsheaf $F' \subseteq F$ which has maximal slope with respect to $H_P$ (e.g. one can take a stable factor in the first Harder-Narasimhan factor of $F$). Without loss of generality, we can assume $F$ and $F'$ are both reflexive. Since $F'$ is also a subsheaf of the trivial bundle $H^0(\Pb_{\da}(H))\otimes \mathcal{O}_{P_n(A)}$, there must be a projection from $H^0(\Pb_{\da}(H))\otimes \mathcal{O}_{P_n(A)}$ to a certain direct summand of it, such that the composition of the embedding and projection $F' \rightarrow H^0(\Pb_{\da}(H))\otimes \mathcal{O}_{P_n(A)} \rightarrow \mathcal{O}_{P_n(A)}$ is non-zero. Since $\mu_{P_n(A)}(F') \geqslant \mu_{P_n(A)}(F) = 0 = \mu_{P_n(A)}(\mathcal{O}_{P_n(A)})$, and $\mathcal{O}_{P_n(A)}$ is also stable with respect to $H_{P}$, the map $F' \rightarrow\mathcal{O}_{P_n(A)}$ must be injective, and its cokernel is supported on a locus of codimension at least $2$. Since both are reflexive, we must have $F' = \mathcal{O}_{P_n(A)}$. Therefore $F$, and consequently $(\mathcal{U}_{\da})_{P}$, have non-trivial global sections. This yields a contradiction.
	
	If $a \geqslant 1$, $F$ would be a subsheaf of the trivial bundle $H^0(\Pb_{\da}(H))\otimes \mathcal{O}_{P_n(A)}$ of positive slope. This yields a contradiction.
	
	Finally, assume $G$ is a reflexive subsheaf of $\mathcal{U}_{\da}$. Then $(G)_{P}$ is an $\mathfrak{S}_{n+1}$-invariant reflexive subsheaf of $(\mathcal{U}_{\da})_{P}$. By the above claim we have $\mu_{H_P}((G)_{P}) < \mu_{H_P}((\mathcal{U}_{\da})_{P})$. It follows from equation \eqref{eqn:MultiCover} that $\mu_{H_K}(G) < \mu_{H_K}(\mathcal{U}_{\da})$. Therefore $\mathcal{U}_{\da}$ is slope stable with respect to $H_K$, as desired.
\end{proof}

\begin{proposition}\label{prop:sameH1}
	There exists some ample class $H' \in \NS(\KA)$ near $H_K$, such that $\mathcal{U}_{\da}$ is $\mu_{H'}$-stable for all $\da \in \dA$ simultaneously.
\end{proposition}

\begin{proof}
By using Proposition \ref{prop:stableEx} instead of Proposition \ref{thm:boundary-stable}, the proof is literally the same as the proof of Proposition \ref{prop:sameH2}, except
that the step which shows that S is a finite set has to be modified. In this situation the slope $c := \mu_\beta(\mathcal{U}_{\da})$ is also independent of the choice of $\da\in \dA$ and we have
	$$ S := \{ c_1(F) \mid F \subseteq \mathcal{U}_{\da} \text{ for some } \da \in \dA \text{ such that } \mu_\beta(F) \geqslant c \}. $$
	Note that $V:=H^0(\Pb_{\da}(H))\cong \CC^{n+r+1}$ is independent of $\da\in \dA$. Since we have $\mathcal{U}_{\da} \subseteq V \otimes \mathcal{O}_{\KA}$ for all $\da \in \dA$ by \eqref{exsequ}, we obtain that $S$ is a subset of
	$$ T := \{ c_1(F) \mid F \subseteq V\otimes \mathcal{O}_{\KA} \text{ such that } \mu_\beta(F) \geqslant c \}, $$
	which is finite by \cite[Theorem 2.29]{greb16}, and hence so is $S$. This concludes the proof.
\end{proof}

\subsection{A component of the moduli space}
We start this subsection by remarking that the $\mathbb{P}^{n-1}$-functor $\Theta$ defined in \eqref{theta} induces an isomorphism of graded vector spaces
\begin{equation}\label{eqn:pnfunc}
	\Ext^{*}_{\KA}(\Theta(E),\Theta(F))\cong \Ext^{*}_A(E,F)\otimes H^{*}(\mathbb{P}^{n-1},\mathbb{C})
\end{equation}
for any $E, F\in \Db(A)$, see \cite[\S 2.1]{add16-2}.


We now turn to the main result of this section. Let $H'$ be an ample class that satisfies Proposition \ref{prop:sameH1}, and let $\mathcal{M}$ the moduli space of $\mu_{H'}$-stable sheaves on $\KA$ with the same numerical invariants as $\mathcal{U}_{\da}$. Then the universal family $\mathcal{U}$ defines a classifying morphism
\begin{equation}\label{eqn:class}
	f \colon \dA \longrightarrow \mathcal{M}, \quad \da\longmapsto \left[ \mathcal{U}_{\da}\right] 
\end{equation}

In fact the morphism $f$ can be described as follows:

\begin{theorem}\label{thm:component1}
	The classifying morphism \eqref{eqn:class} defined by the family $\mathcal{U}$ identifies $\dA$ with a smooth connected component of $\mathcal{M}$.
\end{theorem}

\begin{proof}
	
	We know $\mathcal{U}_{\da}=\Theta(\Pb_{\da}(H))$, so for $\da_1\neq \da_2$ we find by \eqref{eqn:pnfunc} that
	\begin{equation*}
		\Hom_{\KA}(\mathcal{U}_{\da_1},\mathcal{U}_{\da_2})\cong\Hom_A(\Pb_{\da_1}(H),\Pb_{\da_2}(H))\cong H^0(A,\Pb_{\da_1}^{\vee}\otimes\Pb_{\da_2})=0,
	\end{equation*}
	where the last step uses \cite[Lemma 9.9]{huy}.
	
	A similar computation shows
	\begin{align*}
		\Ext^1_{\KA}(\mathcal{U}_{\da},\mathcal{U}_{\da})&\cong\Ext^1_A(\Pb_{\da}(H),\Pb_{\da}(H))
	\cong\Ext^1_{\dA}(\mathcal{O}_{\da},\mathcal{O}_{\da}) \cong T_{\da}\dA.
	\end{align*}
	Here we use 	$\Ext^1_A(\Pb_{\da}(H),\Pb_{\da}(H))\cong\Ext^1_A(\Pb_{\da},\Pb_{\da})$, $\Pb_{\da}\cong\widehat{\Phi}(\mathcal{O}_{\da})$ and the fact that $\widehat{\Phi}$ is an equivalence from $\Db(\dA)$ to $\Db(A)$.
	
	These computations imply that $f$ is injective on closed points and that we have $\dim(T_{[\mathcal{U}_{\da}]}\mathcal{M})=2$ for all $\da\in \dA$. The claim then follows from an argument similar to the proof of Theorem \ref{thm:firstmain}. 
\end{proof}


\begin{thebibliography}{10}
	
	\bibitem{add}
	Nicolas Addington.
	\newblock New derived symmetries of some hyperk\"{a}hler varieties.
	\newblock {\em Algebr. Geom.}, 3(2):223--260, 2016.
	
	\bibitem{add16-2}
	Nicolas Addington, Will Donovan, and Ciaran Meachan.
	\newblock Moduli spaces of torsion sheaves on {K}3 surfaces and derived
	equivalences.
	\newblock {\em J. Lond. Math. Soc. (2)}, 93(3):846--865, 2016.
	
	\bibitem{beau}
	Arnaud Beauville.
	\newblock Vari\'{e}t\'{e}s {K}\"{a}hleriennes dont la premi\`ere classe de
	{C}hern est nulle.
	\newblock {\em J. Differential Geom.}, 18(4):755--782 (1984), 1983.
	
	\bibitem{lange}
	Christina Birkenhake and Herbert Lange.
	\newblock The dual polarization of an abelian variety.
	\newblock {\em Arch. Math. (Basel)}, 73(5):380--389, 1999.
	
	\bibitem{britze}
	Michael Britze.
	\newblock {\em On the Cohomology of Generalized Kummer Varieties}.
	\newblock PhD thesis, Universit{\"a}t zu K{\"o}ln, 2002.
	
	\bibitem{dCM}
	Mark Andrea~A. de~Cataldo and Luca Migliorini.
	\newblock The hard {L}efschetz theorem and the topology of semismall maps.
	\newblock {\em Ann. Sci. \'{E}cole Norm. Sup. (4)}, 35(5):759--772, 2002.
	
	\bibitem{EH16}
	David Eisenbud and Joe Harris.
	\newblock {\em 3264 and all that---a second course in algebraic geometry}.
	\newblock Cambridge University Press, Cambridge, 2016.
	
	\bibitem{friedman}
	Robert Friedman.
	\newblock {\em Algebraic surfaces and holomorphic vector bundles}.
	\newblock Universitext. Springer-Verlag, New York, 1998.
	
	\bibitem{greb16}
	Daniel Greb, Stefan Kebekus, and Thomas Peternell.
	\newblock Movable curves and semistable sheaves.
	\newblock {\em Int. Math. Res. Not. IMRN}, 2:536--570, 2016.
	
	\bibitem{gul}
	Martin~G. Gulbrandsen.
	\newblock Lagrangian fibrations on generalized {K}ummer varieties.
	\newblock {\em Bull. Soc. Math. France}, 135(2):283--298, 2007.
	
	\bibitem{hart}
	Robin Hartshorne.
	\newblock {\em Algebraic geometry}.
	\newblock Springer-Verlag, New York-Heidelberg, 1977.
	\newblock Graduate Texts in Mathematics, No. 52.
	
	\bibitem{hashi}
	Mitsuyasu Hashimoto.
	\newblock Base change of invariant subrings.
	\newblock {\em Nagoya Math. J.}, 186:165--171, 2007.
	
	\bibitem{huy}
	Daniel Huybrechts.
	\newblock {\em Fourier-{M}ukai transforms in algebraic geometry}.
	\newblock Oxford Mathematical Monographs. The Clarendon Press, Oxford
	University Press, Oxford, 2006.
	
	\bibitem{huy2}
	Daniel Huybrechts.
	\newblock {\em Lectures on {K}3 surfaces}, volume 158 of {\em Cambridge Studies
		in Advanced Mathematics}.
	\newblock Cambridge University Press, Cambridge, 2016.
	
	\bibitem{HL}
	Daniel Huybrechts and Manfred Lehn.
	\newblock {\em The geometry of moduli spaces of sheaves}.
	\newblock Cambridge Mathematical Library. Cambridge University Press,
	Cambridge, second edition, 2010.
	
	\bibitem{lehn}
	Manfred Lehn.
	\newblock Symplectic moduli spaces.
	\newblock In {\em Intersection theory and moduli}, ICTP Lect. Notes, XIX, pages
	139--184. Abdus Salam Int. Cent. Theoret. Phys., Trieste, 2004.
	
	\bibitem{markman}
	Eyal Markman.
	\newblock {Stable vector bundles on a hyper-Kahler manifold with a rank 1
		obstruction map are modular}.
	\newblock {\em arXiv:2107.13991}, pages 1--88, 2021.
	
	\bibitem{meac}
	Ciaran Meachan.
	\newblock Derived autoequivalences of generalised {K}ummer varieties.
	\newblock {\em Math. Res. Lett.}, 22(4):1193--1221, 2015.
	
	\bibitem{muk2}
	Shigeru Mukai.
	\newblock Duality between {$D(X)$} and {$D(\hat X)$} with its application to
	{P}icard sheaves.
	\newblock {\em Nagoya Math. J.}, 81:153--175, 1981.
	
	\bibitem{muk}
	Shigeru Mukai.
	\newblock Fourier functor and its application to the moduli of bundles on an
	abelian variety.
	\newblock In {\em Algebraic geometry, {S}endai, 1985}, volume~10 of {\em Adv.
		Stud. Pure Math.}, pages 515--550. North-Holland, Amsterdam, 1987.
	
\bibitem{ogrady}
Kieran~G. O'Grady.
\newblock Modular sheaves on hyperk\"{a}hler varieties.
\newblock {\em Algebr. Geom.}, 9(1):1--38, 2022.
	
	\bibitem{rz}
	Fabian Reede and Ziyu Zhang.
	\newblock Examples of smooth components of moduli spaces of stable sheaves.
	\newblock {\em Manuscripta Math.}, 165(3-4):605--621, 2021.
	
\bibitem{rz2}
Fabian Reede and Ziyu Zhang.
\newblock Stability of some vector bundles on {H}ilbert schemes of points on
{K}3 surfaces.
\newblock {\em Math. Z.}, 301(1):315--341, 2022.
	
	\bibitem{stapleton}
	David Stapleton.
	\newblock Geometry and stability of tautological bundles on {H}ilbert schemes
	of points.
	\newblock {\em Algebra \& Number Theory}, 10(6):1173--1190, 2016.
	
	\bibitem{wandel}
	Malte Wandel.
	\newblock Tautological sheaves: stability, moduli spaces and restrictions to
	generalised {K}ummer varieties.
	\newblock {\em Osaka J. Math.}, 53(4):889--910, 2016.
	
	\bibitem{wray}
	Andrew Wray.
	\newblock {\em Moduli {S}paces of {H}ermite-{E}instein {C}onnections over {K}3
		{S}urfaces}.
	\newblock PhD thesis, University of Oregon, 2020.
	
	\bibitem{yosh}
	K\={o}ta Yoshioka.
	\newblock Moduli spaces of stable sheaves on abelian surfaces.
	\newblock {\em Math. Ann.}, 321(4):817--884, 2001.
	
\end{thebibliography}
\end{document}